\documentclass{amsart}
\usepackage{color}

\newcommand{\NN}{\ensuremath{\mathbb{N}}}

\newcommand{\RR}{\ensuremath{\mathbb{R}}}

\newcommand{\TT}{\ensuremath{\mathbb{T}}}

\newcommand{\ZZ}{\ensuremath{\mathbb{Z}}}

   \newtheorem{lemma}{Lemma}[section]
   \newtheorem{theorem}[lemma]{Theorem}
   \newtheorem{remark}[lemma]{Remark}

   \newtheorem{definition}[lemma]{Definition}
\numberwithin{equation}{section}
\newcommand{\R}{{\mathbb R}}

\newcommand{\um}{{u^\mu}}

\begin{document}


\title[On 3D Navier-Stokes equations: regularization and uniqueness by delays.]
{On 3D Navier-Stokes equations: \\regularization and uniqueness by delays. }

\author{Hakima Bessaih}
 \address[Hakima Bessaih]{Department of  Mathematics\\
 University of Wyoming\\
Laramie 82071 USA}
\email[Hakima Bessaih]{bessaih@uwyo.edu}

\author{Mar\'{\i}a J. Garrido-Atienza}\address[Mar\'{\i}a J. Garrido-Atienza]{Dpto. Ecuaciones Diferenciales y An\'alisis Num\'erico\\
Facultad de Matem\'aticas, Universidad de Sevilla, Avda. Reina Mercedes, s/n, 41012-Sevilla, Spain} \email[Mar\'{\i}a J. Garrido-Atienza]{mgarrido@us.es}

\author{Bj{\"o}rn Schmalfu{\ss }}
\address[Bj{\"o}rn Schmalfu{\ss }]{Institut f\"{u}r Stochastik\\
Friedrich Schiller Universit{\"a}t Jena, Ernst Abbe Platz 2, 77043\\
Jena,
Germany
 }
\email[Bj{\"o}rn Schmalfu{\ss }]{bjoern.schmalfuss@uni-jena.de}

\begin{abstract} A modified version of the three dimensional Navier-Stokes equations is considered with periodic boundary conditions. A bounded constant delay is introduced into the convective term, that produces a regularizing effect on the solution. In fact, by assuming appropriate regularity on the initial data, the solutions of the delayed equations are proved to be regular and, as a consequence, existence and also uniqueness of a global weak solution is obtained. Moreover, the associated flow is constructed and the continuity of the semigroup is proved. Finally, we investigate the passage to the limit on the delay, obtaining that the limit is a weak solution of the Navier-Stokes equations.
\end{abstract}

\maketitle
\maketitle

\section*{}

\section{Introduction}\label{s1} The incompressible Navier-Stokes equations are described by time evolution of the velocity $u$ in a bounded or unbounded domain of $\mathbb{R}^n, \ n=2, 3$  and are given by:
\begin{align*}
 u^\prime(t,x) & +(u(t,x)\cdot\nabla) u(t,x)-\nu\Delta u(t,x)+\nabla p(t,x)=0, \\
   & {\rm div}\, u(t,x)=0,\quad u(0,x)=u_0(x),
\end{align*}
where $\nu>0$ is the viscosity of the fluid, $p$ denotes the pressure and $u_0(x)$ denotes the initial datum. The existence and uniqueness of solutions is known to hold for $n=2$ while only partially solved in the case $n=3$.  In fact,  existence of global weak solutions is known since the seminal work of Leray \cite{leray},  while their uniqueness and regularity is still an open question.  The difficulty in dimension 3 comes from the nonlinear term $(u\, \cdot\, \nabla) u$ that  requires more regularity.
 In fact, this regularity is not satisfied by the energy estimates while it is in dimension 2.
 In particular, the lack of this regularity is essentially the reason for which the uniqueness  cannot  be proved for weak solutions. \\

 Many regularizations have been considered into the nonlinear term to overcome this difficulty.  Depending on the way the regularization is introduced, yields to a wide range of models. We refer to \cite{FHT} for the so-called Camassa-Holm equations, to \cite{CHOT, OT} for the Leray-alpha models  and the references therein. Another way of regularizing the equation is to add a damping term consisting of a monotone operator, see for example \cite{BTZ, KZ, MTT}. A different model that takes into account some physical aspects are the rotating flows, see for e.g. \cite{BMN, CDGG-00, CDGG-06}. \\

 In this paper, we will introduce a delay in the nonlinear term $(u\, \cdot\, \nabla) u$. We consider the following modified version of the 3D Navier-Stokes equations with constant delay $\mu>0$:
\begin{align}\label{delay-i}
\begin{split}
  u^\prime(t,x) & +(u(t-\mu, x)\cdot\nabla)u(t,x)-\nu\Delta u(t,x)+\nabla p(t,x)=f(x), \\
   & {\rm div}\, u(t,x)=0,\quad u(0,x)=u_0(x),\quad u(\tau,x)=\phi(\tau,x),\quad \tau\in [-\mu,0).
\end{split}
\end{align}
As we will prove later, this delay introduces a regularizing  effect in the equations and allows to prove the uniqueness of weak solutions.
In \cite{GP-2011}, with a time variable bounded  delay function in the convective term, the authors investigated the existence of weak solutions and the existence and uniqueness of strong solutions with their longtime behavior. For a more general delay (infinite delay of convolutional type) the same authors investigated mild solutions in \cite{GP-2015},  however, these solutions are only local in time. In \cite{V-08}, Varnhorn considered a similar delay used in equation $\eqref{delay-i}$ and investigated the existence and uniqueness of strong solutions in a bounded and regular domain $D\subset\mathbb{R}^3$ with Dirichlet boundary conditions. In that article, the initial delay function  $\chi$ is assumed to satisfy that $\chi\in C([-\mu,0],H^2(D))$,  with $\frac{\partial \chi}{\partial t} \in C([-\mu,0],L_2(D))$ and ${\rm div}\  \chi=0$. With these assumptions, the author proved the existence and uniqueness of strong solutions. Our setting is different since we are dealing with weak solutions. Indeed, our initial function $\chi=(\phi,u_0)\in L_2(-\mu,0,V^{1+\alpha})\times V^\alpha$ with $\alpha>1/2$ (see the functional setting in Section \ref{s2}) and our domain has periodic boundary conditions. This choice makes possible to apply Lemma \ref{l2} which is specific for periodic boundary conditions. As we will see later, the use of this lemma in  the nonlinear term in (\ref{delay-i}) is crucial to get better estimates that improve the regularity of solutions.\\

In this paper, our focus will be on the existence and uniqueness of global weak solutions, although our method would include dealing with strong solutions when $\alpha\geq 1$ (as we said above, throughout the paper $\alpha>1/2$). The main ingredient to establish it,  as we mentioned earlier, is to use the regularizing effect of the delay on the convective term. Indeed, first we will investigate the linearized version \eqref{eq1} of system \eqref{delay-i}. This equation comes naturally when investigating  the system on the interval
$[0,\mu] $. We prove existence and uniqueness of weak solutions, then we establish that these solutions are  more regular and are in the spaces $V^\alpha$ (see Section \ref{s3} for details). Using the linearized equation, we are able to construct weak solutions for equation \eqref{delay-i} by glueing the solutions obtained on each interval $[0,\mu], [\mu, 2\mu],\dots $ and so on. Each solution is obtained from the previous step and uses the linearized construction from Section \ref{s3}. With this procedure, we are also able to prove the existence of a continuous semigroup $S(t)$, see Section  \ref{s5} for details. The semigroup theory would allow to study the longtime behavior in terms of global attractors, invariant manifolds, and so on, however these topics are not investigated in the current paper and will be considered in a future research.

Furthermore, by passing to the limit on the delay, we prove that our regularized solutions converge to a weak solution of  the 3D Navier-Stokes equations.  This passage to the limit on the delay was also addressed by Varnhom in  \cite{V-08}. \\

The contain of the paper is as follows: in Section \ref{s2} we introduce the functional spaces and one key result related to a suitable estimate of the trilinear form in the Navier--Stokes equations. In Section \ref{s3}, the linearized version of (\ref{delay-i}) is investigated. Existence and uniqueness of a solution as well as its regularity are obtained. The Navier--Stokes equations with constant delay are addressed in the rest of the paper. In Section \ref{s4} we construct the weak solution by using a concatenation procedure based on the linearized equations. The generation of a continuous semigroup in the phase space $V^\alpha$ is obtained in Section \ref{s5} thanks to the uniqueness of the weak solution. Finally, in Section \ref{s6} we address the passage to the limit when the delay $\mu$ goes to zero, obtaining in the limit the Navier--Stokes equations.

\section{Functional Setting}\label{s2}
First of all, we introduce the functional setting in which our investigations will be carried out. For a more exhaustive description of the setting, we refer the reader to \cite{Fursikov}, Chapter 3, from Page 152.

Let us consider the torus $\TT^3_L$ in $\R^3$ of length $L$ given by the set $\{ x=(x_1,x_2,x_3)\in \RR^3 \,:\, -L/2\leq x_i \leq L/2, \, x_i=-L/2 \mbox{ is identified with  } x_i=L/2, \, i=1,2,3\}$. Consider $L$-periodic functions $\psi(x)$ that can be expanded into Fourier series
$$\psi(x)=\sum_{\zeta\in \ZZ^3_L} e^{i(x,\zeta)} \hat \psi(\zeta),$$
where
$$\ZZ^3_L=\{ \zeta=(\zeta_1,\zeta_2,\zeta_3)\,:\, \zeta_i=2\pi k_i/L,\, k_i \mbox { is an integer}, \, i=1,2,3\},$$
and
$$\hat \psi(\zeta)=L^{-3} \int_{\TT^3_L} e^{-(y,\zeta)}\psi(y)dy$$
denote the Fourier coefficients of $\psi$.

For $s\in \RR$, we denote by $H^s(\TT^3_L)$ the Sobolev space of $L$--periodic functions such that $\hat \psi(\zeta)=\overline{\hat \psi(-\zeta)}$ equipped with the norm
$$\|\psi\|_s= \bigg(\sum_{\zeta\in \ZZ^3_L} (1+|\zeta|^2)^s |\hat \psi(\zeta)|^2\bigg)^\frac12.$$
When $\hat \psi(0)=0$ the corresponding subspace is denoted by $\dot H^s(\TT^3_L)$ with equivalent norm
$$\bigg(\sum_{\zeta\in \ZZ^3_L\setminus \{0\}} |\zeta|^{2s} |\hat \psi(\zeta)|^2\bigg)^{1/2}.$$
In particular, these spaces are Hilbert--spaces
with the inner product
\begin{equation*}
  (\psi_1,\psi_2)_s =\sum_{\zeta\in \ZZ^3_L\setminus \{0\}}|\zeta|^{2s} \psi_1(\zeta)\overline{\hat\psi_2(\zeta)}.
\end{equation*}
%For $s=0$ this is the space $L_2(\TT_L^3)^3$.
\medskip
We denote by $\dot{\mathbb H}^s(\mathbb T_L^3)=\dot H^s(\TT^3_L)^3$ and introduce the spaces
\begin{align*}
V^1=&\{u\in \dot{\mathbb H}^1(\mathbb T_L^3), \, \rm{div}\ u=0 \},\\
V^0=&\{u\in \dot{\mathbb H}^0(\mathbb T_L^3), \, \rm{div}\ u=0 \},\\
V^{-1}=&\{u\in \dot{\mathbb H}^{-1}(\mathbb T_L^3)=(\dot{\mathbb H}^1(\mathbb T_L^3))^\prime, \, \rm{div}\ u=0 \}.
\end{align*}
Then $V^{-1}$ is the dual space of $V^1$ and $V^1\subset V^0 \subset V^{-1}$ where the injections are continuous and each space is dense in the following one. We shall denote by $(\cdot,\cdot)$ the scalar product in $V^0$.

We introduce the Stokes operator $A$ as in \cite{temamP}, Section 2.2. Page 9, with domain of $A$ given by
$$D(A)=\{u\in V^0, \, \Delta u\in V^0 \}.$$
Then for the periodic boundary conditions we have
\begin{equation*}
  Au=-\Delta u.
\end{equation*}
The operator $A$ can be seen as an unbounded positive linear selfadjoint operator on $V^0$, and we can define the powers $A^s$, $s\in \RR$ with domain $D(A^s)$. We set $V^{s}=D(A^{s/2})$, that is a closed subspace of $\dot {\mathbb H}^s(\TT_L^3)$, then
$$V^s=\{u\in \dot{\mathbb H}^s(\mathbb T_L^3), \, \rm{div}\ u=0 \}$$
and the norms $\|A^{s/2}u\|_0$ and $\|u\|_s$ are equivalent on $V^s$.  The operator $A$ defines an isomorphism from $V^s$ to $V^{s-2}$. In addition, $A$ has a positive countable spectrum of finite multiplicity $0<\lambda_1\le \lambda_2\le \cdots$, $\lambda_j \to \infty$, and the associated eigenvectors $e_1,e_2,\cdots$ form a complete orthogonal system in $V^s$.

When $s_1<s_2$,  the embedding $V^{s_2}\subset V^{s_1}$ is compact and dense. The space $V^{-s}$ is the dual space of $V^s$ for $s\in\RR$, see Temam \cite{temamP}, from page 9.

We will denote by $\langle \cdot, \cdot \rangle$ the duality product between $V^s$ and $V^{-s}$ no matter the value of $s\in \RR$.\\

Let us consider the trilinear form which describes one of the coefficients of the Navier-Stokes equations:

\begin{equation*}
  b(u,v,w)=\sum_{i,j=1}^3\int_{\TT_L^3} u_j \frac{\partial v_i}{\partial x_j}w_idx.
\end{equation*}

To investigate this trilinear form we use the following multiplication lemma, proved in Fursikov \cite{Fursikov}, Lemma 3.4.4, page 153 \footnote{Note that this result holds for more general periodic functions than those of $H^s(\TT^3_L)$, see Remark \ref{remark1} below.}:
\begin{lemma}\label{l1}
Let $r_3 \geq 0$. For $\psi \in H^{r_1}(\TT_L^3)$ and $\varphi \in H^{r_2}(\TT_L^3)$, there exists a constant $c>0$ such that
$$\| \psi \varphi \|_{r_3} \leq c \|\psi\|_{r_1} \|\varphi\|_{r_2},$$
if either $r_1>r_3$, $r_2>r_3$, $r_1+r_2 \geq r_3+3/2$ or $r_1\geq r_3$, $r_2\geq r_3$, $r_1+r_2 > r_3+3/2$.
\end{lemma}

As a consequence, we can establish the following result concerning the domain of $b$:

\begin{lemma}\label{l2}
The trilinear form $b$ can be continuously  extended to $V^{s_1}\times V^{s_2+1}\times V^{s_3}$  for $s_i\in\RR$ if
either $s_i+s_j\ge 0$ for $i\not=j$, $s_1+s_2+s_3>3/2$ or $s_i+s_j> 0$ for $i\not=j$, $s_1+s_2+s_3\ge 3/2$. Therefore, under either of the previous settings, there exists a constant $c$ depending only on $s_i$ such that
\begin{equation*}
  |b(u,v,w)|\le c\|u\|_{s_1}\|v\|_{s_2+1}\|w\|_{s_3}
\end{equation*}
for $u\in V^{s_1},\,v\in V^{s_2+1},\,w\in V^{s_3}$.
\end{lemma}

\begin{proof}
We consider only the case in which $s_i+s_j\ge 0$ for $i\not=j$, with $s_1+s_2+s_3>3/2$, since the proof for the other case follows analogously.

Note that if the $s_i\geq 0$, $i=1,2,3$, then the result follows by \cite{temamP}, Lemma 2.1, Page 12 (see also \cite{Fursikov}, Lemma 4.4.6, Page 158 for the homogeneous Dirichlet boundary conditions case).

Hence, assume that there exists only one $i*\in \{1,2,3\}$ such that $s_{i*}<0$ (the uniqueness comes from the fact that $s_i+s_j\ge 0$ for $i\not=j$). In a first step, let us consider trigonometric polynomials functions $(u_1,u_2,u_3)\in  V^{s_1}\times V^{s_2+1}\times V^{s_3}$. Then the integral of the product $u_1 u_2u_3$ exists. Indeed, using a general notation, we obtain
\begin{align*}
\bigg| \int_{\TT_L^3} u_1 u_2u_3 dx \bigg|&=\bigg| \int_{\TT_L^3} u_{i*} u_j u_k dx \bigg|\\
&\leq \|u_{i*}\|_{s_{i*}} \|u_j u_k\|_{-s_{i*}} \leq c \|u_{i*}\|_{s_{i*}} \|u_j\|_{s_j}\|u_k\|_{s_{k}}
\end{align*}
where in the last inequality we have taken $r_1=s_j$, $r_2=s_k$ and $r_3=-s_{i*}$, that according to the assumptions, are such that $r_1\geq r_3$, $r_2\geq r_3$ and $r_1+r_2>3/2+r_3$, so that we can apply Lemma \ref{l1}. The proof is therefore complete, since one of the above factors corresponds to a derivative, hence this yields to the +1 that appears in the statement of Lemma \ref{l2}.
Finally, it suffices to remind that the trigonometric polynomials functions are dense in the spaces $V^s$, hence the proof is complete.
\end{proof}

\begin{remark}\label{remark1}
Similar results were proved by Fursikov \cite{Fursikov} when considering a bounded domain $\Omega \subset \RR^3$, $\partial \Omega \in C^\infty$ with homogeneous Dirichlet conditions, but with more restrictive assumptions. In Lemma 3.4.2, Page 150, with the additional assumptions $s_i\geq 0$, and in Lemma 3.4.6, Page 157, under the additional requirement that the parameters $s_i>-1/2$ for $i=1,2,3$. In fact, these last constraints ensure that the spaces $H^{s/2}$ and $V^s$ coincide in the homogeneous Dirichlet conditions case. On the other hand, in the periodic boundary setting, for a similar result as Lemma \ref{l2} see also \cite{temamP}, Lemma 2.1, Page 12, where the additional assumptions are $s_i\geq 0$.
\end{remark}

From $b$ we can derive a bilinear operator $B:V^{s_1}\times V^{s_2+1} \rightarrow V^{-s_3}$ given by
\begin{align*}
& \langle B(u,v),w\rangle = b(u,v,w),
\end{align*}
such that
\begin{align}\label{eq7}
\begin{split}
& \|B(u,v)\|_{-s_3}\le c_B\|u\|_{s_1}\|v\|_{s_2+1}
\end{split}
\end{align}
with $s_1,\,s_2,\,s_3$ satisfying the conditions of  Lemma \ref{l2}. Note that in the following we shall denote the constant $c_B$ simply by $c$, a positive constant that may change from line to line.
\medskip

The following property of $b$ is well-known:

\begin{lemma}\label{l5}
Suppose that $u,\,v,\,w\in V^1$. Then we have $ b(u,v,w)=-b(u,w,v)$, hence
$$b(u,v,v)=0.$$
\end{lemma}

Finally we mention that for $\mu>0$ the spaces $L_\infty(0,\mu,V^\alpha),\,L_2(0,\mu,V^\alpha),\,C([0,\mu],V^\alpha)$ and $C^\beta([0,\mu],V^\alpha),\,\beta\in (0,1)$, have
the usual meanings.

\section{The 3D linearized Navier--Stokes equations}\label{s3}

In the sequel, we consider the 3D linearized Navier--Stokes equations with periodic boundary conditions over the torus  $\TT_L^3$ in $\RR^3$
\begin{align}\label{eq1}
\begin{split}
  u^\prime(t,x) & +(\psi(t),\nabla )u(t)-\nu\Delta u(t)+\nabla p(t,x)=f(x), \\
   & {\rm div}\, u(t,x)=0,\quad u(0)=u_0(x).
\end{split}
\end{align}
Let us emphasize that these equations are a simpler version of the 3D Navier--Stokes equations, since the term $(u,\nabla)u$ has been replaced by $(\psi,\nabla)u$, where $\psi$ will have a suitable regularity. As we will see below, the analysis of the existence, uniqueness and regularity of solutions to (\ref{eq1}) will play an important role to further deal with the delayed Navier--Stokes equations.

\subsection{Existence and Uniqueness of weak solutions}

Assume from now on that $\alpha$ is a real parameter such that $\alpha>1/2$. In virtue of the well-known Helmholtz-projection, which is commonly applied to the Navier-Stokes equations in order to eliminate the pressure, the meaning of a weak solution to (\ref{eq1}) is understood as follows:

\begin{definition}\label{d1}
For $\alpha>1/2$, let $\mu>0$, $\psi$ be a fixed element in $L_2(0,\,\mu,V^{1+\alpha})$, $u_0\in V^0$ and $f\in V^{-1}$.
We say that $u$ is a weak solution to \eqref{eq1} on $(0,\mu)$ if $u\in L_2(0,\mu,V^1)$ and
\begin{equation}\label{eq2}
  \frac{d}{dt}\langle u(t),v\rangle=\langle f-\nu Au(t)-B(\psi(t),u(t)),v\rangle,\qquad u(0)=u_0,
\end{equation}
for any $v\in V^1$, or
\begin{equation*}
  \frac{du(t)}{dt}=f-\nu Au(t)-B(\psi(t),u(t)),\qquad u(0)=u_0,
\end{equation*}
as an equation in $V^{-1}$.
\end{definition}
Notice that the last definition makes sense because the right hand side belongs to $L_1(0,\mu,V^{-1})$. Indeed we have $Au(\cdot)\in L_2(0,\mu,V^{-1})$ since
\begin{equation}\label{e1}
 \int_0^\mu \|Au(r)\|_{-1}^2dr=  \int_0^\mu \|A^{-1/2} Au(r)\|_{0}^2dr= \int_0^\mu \|u(r)\|_{1}^2dr<\infty,
 \end{equation}
and, thanks to \eqref{eq7},
\begin{equation*}
 \|B(\psi(t),u(t))\|_{-1}\le c  \|\psi(t)\|_{1}\|u(t)\|_1.
\end{equation*}
%Furthermore, such regularity of the derivative gives sense to the initial condition.
\bigskip

The time derivative in the above definition has to be understood in the distributional sense, hence multiplying \eqref{eq2} by a test function
$\varphi\in C_0^\infty([0,\mu])$, by integration we obtain
\begin{align}\label{eq3}
\begin{split}
  -\int_0^\mu \langle u(r),v \rangle \varphi^\prime(r)dr & +\nu\int_0^\mu\langle A^\frac12 u(r),A^\frac12 v\rangle\varphi(r)dr +\int_0^\mu b(\psi(r),u(r),v)\varphi(r)dr\\
  =& \int_0^\mu\langle f,v\rangle \varphi(r)dr.
\end{split}
\end{align}

\begin{lemma}\label{l3}
For $\alpha>1/2$, assume that $\psi \in L_2(0,\mu,V^{1+\alpha})$. Then the equation \eqref{eq2} has a (weak) solution in $L_\infty(0,\mu,V^0)\cap L_2(0,\mu,V^1)\cap C^\gamma([0,\mu],V^{-s}),\, 0\leq \gamma\leq 1/2,\,s\ge 1$. Moreover, $u$ is weakly continuous in $V^0$.
\end{lemma}

\begin{proof}
Let denote by $\pi_m$ the projection on the linear subspace of $V^0$ given by the span of the eigenvectors $e_1,\cdots,e_m$, which are supposed to have norm one in $V^0$.  Consider the Galerkin--approximations $u_m(t)\in \pi_mV^0$ to \eqref{eq2}, then
 \begin{equation*}
  u_m^\prime(t)+\nu Au_m(t)+ \pi_m B(\psi(t),u_m(t))=\pi_m f,\quad u_m(0)=\pi_m u_0.
\end{equation*}
By standard methods we obtain the  following a priori estimate
\begin{equation*}
  \|u_m\|_{L_\infty(0,\mu,V^0)}^2+\nu\|u_m\|_{L_2(0,\mu,V^1)}^2\le \|u_0\|_0^2+\frac{1}{\nu}\|f\|_{-1}^2.
\end{equation*}
In addition, for $s\ge 1$ and $0\leq \tau<t\leq \mu$, we have
\begin{align*}
  \|u_m(t)-u_m(\tau)\|_{-s}\le&\nu (t-\tau)^\frac12\bigg(\int_0^\mu\|Au_m(r)\|_{-s}^2dr\bigg)^\frac12\\
  &+(t-\tau)^\frac12\bigg(\int_0^\mu\|\pi_m B(\psi(r),u_m(r))\|_{-s}^2dr\bigg)^\frac12.
\end{align*}
On the other hand,
\begin{equation*}
  \int_0^\mu\|Au_m(r)\|_{-s}^2dr=\int_0^\mu\|A^{-s/2}Au_m(r)\|_{0}^2dr=\int_0^\mu\|u_m(r)\|_{-s+2}^2dr\le c\|u_m\|_{L_2(0,\mu,V^1)}^2,
\end{equation*}
for an appropriate embedding constant $c$ (note that $V^1\subseteq V^{-s+2}$ because $s\geq 1$). Furthermore, by Lemma \ref{l1}, with the choice $s_1=1+\alpha$, $s_2=-1,\,s_3=s\ge 1>3/2-\alpha$ (where this last inequality follows by the assumption $\alpha >1/2$), we obtain
\begin{align*}
  \|B(\psi(r),u_m(r))\|_{-s}\le c\|\psi(r)\|_{1+\alpha}\|u_m(r)\|_{0}.
\end{align*}
Hence the sequence $(u_m)_{m\in\NN}$ is uniformly bounded in $C^\frac12([0,\mu],V^{-s})\cap L_2(0,\mu,V^1)$.
By the Dubinskij-theorem, see Vishik and Fursikov \cite{FurVis} Chapter 4, Theorem 4.1,  $(u_m)_{m\in\NN}$ is relatively compact in $L_2(0,\mu,V^0)\cap C([0,\mu],V^{-s})$, weak relatively compact in $L_2(0,\mu,V^1)$ and weak star relatively compact in $L_\infty(0,\mu,V^0)$. Standard arguments now allow us to conclude that the limit point of this sequence satisfies \eqref{eq3}
in a weak sense. Indeed, the following equations are equivalent to the Galerkin approximations
\begin{align}\label{gal}
\begin{split}
  -\int_0^t \langle u_m(r),v \rangle \varphi^\prime(r)dr+&\nu\int_0^t \langle A^{1/2}u_m(r),A^{1/2}v\rangle\varphi(r)dr+\int_0^t \pi_m b(\psi(r),u_m(r),v)\varphi(r)dr \\
  &=(\pi_m u_0, v)\varphi(0)-(u_m(t),v)\varphi(t)+\int_0^t \langle \pi_m f,v\rangle \varphi(r)dr,
  \end{split}
\end{align}
for $v\in V^s$ and $\varphi\in C^\infty([0,t])$, for every $t\in [0,\mu]$.
In particular, there exists a subsequence, also denoted by $(u_m)_{m\in \NN}$ and an element $u\in L_\infty(0,\mu,V^0)\cap L_2(0,\mu,V^1)\cap C([0,\mu],V^{-s})$ such that
\begin{align*}
  &\lim_{m\to \infty}\|u_m-u\|_{L_2(0,\mu,V^0)}=0,\\
  &\lim_{m\to \infty}\|u_m-u\|_{C([0,\mu],V^{-s})}=0,\\
  w-&\lim_{m\to \infty}u_m=u\quad \text{in }L_2(0,\mu,V^1),\\
 w-\ast &\lim_{m\to \infty}u_m=u\quad \text{in }L_\infty(0,\mu,V^0),\\
  &\lim_{m\to\infty}\|u_m(t)-u(t)\|_{-s}=0,\quad t\in [0,\mu].
\end{align*}
The last limit shows in particular that $u(0)=u_0$. In addition, it follows easily
\begin{align*}
&\lim_{m\to\infty}\int_0^t \langle u_m(r),v \rangle \varphi^\prime(r)dr=\int_0^t \langle u(r),v \rangle \varphi^\prime(r)dr\\
  &\lim_{m\to\infty} \int_0^t \langle A^\frac12 u_m(r),A^\frac12 v\rangle\varphi(r)dr=\int_0^t \langle A^\frac12 u(r),A^\frac12 v\rangle \varphi(r)dr, \\
  &\lim_{m\to\infty} \int_0^t \pi_m b(\psi(r),u_m(r),v)\varphi(r)dr=\int_0^t  b(\psi(r),u(r),v)\varphi(r)dr.
\end{align*}
In the last limit we have used that for $v\in \pi_m V^s$
\begin{equation*}
  | b(\psi(r),u_m(r),v)\varphi(r)- b(\psi(r),u(r),v)\varphi(r)|\le \|\psi(r)\|_{1+\alpha}\|u_m(r)-u(r)\|_0\|v\|_s|\varphi(r)|
\end{equation*}
for a.e. $t\in [0,\mu]$  which follows by Lemma \ref{l2}.

Finally, the weak continuity of $u$ in $V^0$ follows by Lemma 1.4, Chapter 3, Page 263 in \cite{Temam}, since in particular $u\in L_\infty(0,\mu,V^0) \cap  C([0,\mu],V^{-s})$ and $V^0 \subset V^{-s}$ with a continuous injection.
 \end{proof}

\begin{remark}
Some of the steps of the previous proof are based on Lemma \ref{l2}, which is valid for boundary periodic conditions. Moreover, as a consequence of (\ref{gal}), taking a test function $\varphi\in C^\infty([0,t])$, the weak solution of (\ref{eq2}) satisfies
\begin{align}\label{gal1}
\begin{split}
  -\int_0^t \langle u(r),v \rangle \varphi^\prime(r)dr & +\nu\int_0^t\langle A^\frac12 u(r),A^\frac12 v\rangle\varphi(r)dr +\int_0^t b(\psi(r),u(r),v)\varphi(r)dr\\
  =& (u_0,v)\varphi(0)-(u(t),v)\varphi(t)+\int_0^t\langle f,v\rangle \varphi(r)dr,
\end{split}
\end{align}
for every $t\in [0,\mu]$.
\end{remark}

In the following lemma, we establish an energy equality which will be essential to prove the uniqueness of solutions to (\ref{eq1}). It relies on the fact that the solution is such that $u^\prime \in L_2(0,\mu,V^{-1})$. This last property is well--known for the 2D Navier Stokes equations, see for example \cite{temamP}, Theorem 3.1, Page 21, but it is unknown for the 3D case, although for our linearized system (\ref{eq1}) holds true.

\begin{lemma}\label{l3b}
Suppose that $u_0\in V^0$ and $\psi \in L_2(0,\mu,V^{1+\alpha})$. Then, any weak solution $u$ to \eqref{eq1} satisfies the energy equality
\begin{equation*}
  \|u(t)\|_0^2+2\nu\int_0^t\|u(r)\|_1^2dr=\|u_0\|^2_0+2\int_0^t\langle f,u(r)\rangle dr
\end{equation*}
for $t\in [0,\mu]$. In addition $u\in C([0,\mu],V^0)$.
\end{lemma}

\begin{proof}
We remark that we already know that $u^\prime \in L_1(0,\mu, V^{-1})$ (see the comment after Definition \ref{d1}). We want to prove now that this derivative is more regular, namely $u^\prime \in L_2(0,\mu, V^{-1})$. The term $Au \in L_2(0,\mu, V^{-1})$ thanks to (\ref{e1}). Moreover,
\begin{equation*}
  \int_0^\mu \|B(\psi(r),u(r))\|_{-1}^2dr\le c\|\psi\|_{L_2(0,\mu,V^{1+\alpha})}^2 \|u\|_{L_\infty(0,\mu,V^0)}^2<\infty,
\end{equation*}
which follows by Lemma \ref{l2} for $s_1=1+\alpha,\,s_2=-1,\,s_3=1$. As a consequence,
\begin{equation*}
  u\in L_2(0,\mu,V^1),\quad u^\prime \in L_2(0,\mu,V^{-1})
\end{equation*}
hence, Lemma 1.2, Chapter 3, Page 260 of \cite{Temam} ensures that $u\in C([0,\mu],V^0)$ and fulfills
\begin{equation*}
\frac{d}{dt}\|u(t)\|_0^2=2\langle \frac{du(t)}{dt},u (t)\rangle,
\end{equation*}
where the time derivatives are given in the distributional sense. This together with Lemma \ref{l5} above imply
\begin{equation}\label{eq9}
  \frac12\frac{d}{dt}\|u(t)\|_0^2+\nu\|u(t)\|_1^2=\langle f,u(t)\rangle.
\end{equation}
Integrating this equality, the energy equality follows easily.

\end{proof}
As an immediate consequence of the previous result, in the next lemma we can establish the uniqueness of weak solutions to (\ref{eq1}).
\begin{lemma}\label{l4}
Suppose that the assumptions of Lemma \ref{l3} hold. Then, there exists only one weak solution to \eqref{eq1}.
\end{lemma}
\begin{proof}
Let $u^1$ and $u^2$ be two weak solutions to \eqref{eq1} corresponding to the same elements $\psi$, $u_0$ and $f$. Then, by Lemma \ref{l3}, $u^1-u^2 \in L_2(0,\mu,V^1)$. Furthermore,
$$\frac{d}{dt}(u^1(t)-u^2(t))=-\nu A(u^1(t)-u^2(t))-B(\psi(t), u^1(t)-u^2(t)),$$
hence, by the same reasoning as Lemma \ref{l3b}, $\frac{d}{dt}(u^1-u^2)\in L_2(0,\mu,V^{-1})$, with energy equality given by
\begin{equation*}
  \|u^1(t)-u^2(t)\|_0^2+2\nu\int_0^t\|u^1(r)-u^2(r)\|_1^2dr=0.
\end{equation*}
This shows the uniqueness of solutions to (\ref{eq1}).
\end{proof}

\subsection{Further regularity properties}

In what follows, assuming more regularity of the initial condition $u_0$ and the external force $f$, we are going to show that we obtain more regularity for the weak solution $u$ of (\ref{eq1}). This extra regularity of $u$ will be essential to further prove the uniqueness of solutions to the delay system (\ref{delay-i}), see Section \ref{s4}.

\begin{lemma}\label{l2b}
Assume that $\psi \in  L_2(0,\mu,V^{1+\alpha})$, $u_0\in V^\alpha$ and $f\in V^{\alpha-1}$. Then the equation \eqref{eq2} has a (weak) solution in $L_\infty(0,\mu,V^\alpha)\cap L_2(0,\mu,V^{1+\alpha})$.
\end{lemma}
\begin{proof}
For the Galerkin approximations $u_m$, considering its scalar product with $A^\alpha u_m(t)$ in the space $V^0$, we obtain
\begin{align*}
\frac{d}{2dt}\|u_m(t)\|_\alpha ^2+\nu\|u_m(t)\|_{\alpha +1}^2\le |b(\psi(t),u_m(t),A^\alpha u_m(t))|+|\langle f, A^\alpha  u_m(t)\rangle|.
\end{align*}

Hence, by choosing in Lemma \ref{l2} $s_1=1+\alpha ,\,s_2=\alpha $ and $s_3=-\alpha $ (the conditions of that lemma hold since $\alpha >1/2$), we have
\begin{align*}\label{eq4}
\begin{split}
  \frac{d}{dt}\|u_m(t)\|_\alpha ^2+2\nu\|u_m(t)\|_{\alpha +1}^2&\le 2c\|\psi(t)\|_{1+\alpha }\|u_m(t)\|_{1+\alpha }\|u_m(t)\|_\alpha +\frac{2}{\nu}\|f\|_{\alpha-1 }^2+\frac{\nu}{2}\|u_m(t)\|_{\alpha +1}^2\\
  &\le \frac{c^2}{\nu}\|\psi(t)\|_{1+\alpha }^2\|u_m(t)\|_{\alpha }^2+\nu\|u_m(t)\|_{\alpha +1}^2+ \frac{2}{\nu}\|f\|_{\alpha-1 }^2
  +\frac{\nu}{2}\|u_m(t)\|_{\alpha +1}^2
\end{split}
\end{align*}
and therefore
\begin{equation}\label{eq4}
  \frac{d}{dt}\|u_m(t)\|_\alpha ^2+\frac{\nu}{2}\|u_m(t)\|_{\alpha +1}^2\le\frac{c^2}{\nu}\|\psi(t)\|_{1+\alpha }^2\|u_m(t)\|_{\alpha }^2+ \frac{2}{\nu}\|f\|_{\alpha-1 }^2.
\end{equation}
By Gronwall's Lemma we first obtain
\begin{equation}\label{eq5}
   \sup_{m\in\NN}\|u_m\|_{L_\infty(0,\mu,V^\alpha)}<\infty.
\end{equation}
Indeed, Gronwall's Lemma gives us the estimate
\begin{equation*}
  \|u_m(t)\|_\alpha^2\le \|u_0\|_\alpha^2e^{\frac{c^2}{\nu}\int_0^t\|\psi(r)\|_{1+\alpha}^2dr}+\frac{2}{\nu}\|f\|_{\alpha-1}^2\int_0^t e^{\frac{c^2}{\nu}\int_\tau^t\|\psi(r)\|_{1+\alpha}^2dr}d\tau
\end{equation*}
where the right hand side is independent of $m$.
Then substituting (\ref{eq5}) into (\ref{eq4}) and integrating we have
\begin{equation*}
   \sup_{m\in\NN}\|u_m\|_{L_2(0,\mu ,V^{1+\alpha})}<\infty.
\end{equation*}
As a consequence, there exists a subsequence of the sequence of Galerkin approximations that is converging to $\tilde u$ weakly-star in $L_\infty(0,\mu, V^\alpha)$ and weakly in $L_2(0,\mu,V^{1+\alpha})$. On the other hand, we know from Lemma \ref{l2} that we can extract a subsequence converging to $u\in L_\infty(0,\mu,V^0)\cap L_2(0,\mu,V^1)$.  Therefore, taking into account the dense embedding of $V^\alpha$ onto $V^0$ (and hence the dense embedding $L_2(0,\mu,V^\alpha)$ onto $L_2(0,\mu,V^0)$) we have $\tilde u=u$ in $L_\infty(0,\mu,V^0)$. Similarly, we can identify $\tilde u$ and $u$ in $L_2(0,\mu,V^1)$.
\end{proof}

Let us now prove  the continuity of $u(t)$ in $V^\alpha$.

\begin{lemma}\label{l5b}
Let $\psi\in L_2(0,\mu,V^{1+\alpha})$, $u_0\in V^\alpha$ and $f\in V^{\alpha-1}$. Then $u \in C([0,\mu],V^\alpha)$.
\end{lemma}
\begin{proof}
We know from Lemma \ref{l2b}  that $u\in L_2(0,\mu,V^{1+\alpha})$. Furthermore, $\frac{d}{dt}u\in L_2(0,\mu,V^{\alpha-1})$, since straightforwardly $Au\in L_2(0,\mu,V^{\alpha-1})$ and, in virtue of Lemma \ref{l2}, taking $s_1=1+\alpha,\,s_2=\alpha-1,\,s_3=1-\alpha$, we obtain
\begin{equation*}
  \int_0^\mu\|B(\psi(r),u(r))\|_{\alpha-1}^2 dr\le c\|\psi\|_{L_2(0,\mu,V^{1+\alpha})}^2\|u\|_{L_\infty(0,\mu,V^{\alpha})}^2.
\end{equation*}
Hence, $A^{\alpha/2} u\in L_2(0,\mu,V^{1})$ and $\frac{d}{dt}A^{\alpha/2}u\in L_2(0,\mu,V^{-1})$.
Now, using Lemma 1.2, Chapter 3, Page 260 of \cite{Temam} we deduce that $A^{\alpha/2} u \in C([0,\mu],V^0)$, which completes the proof.
\end{proof}

Let $\mathcal{X}^\mu_\alpha$ be the Hilbert space $L_2(0,\mu,V^{1+\alpha})\times V^\alpha$.

\begin{theorem}\label{t3}
Consider the mapping
\begin{equation*}
  U:\mathcal{X}^\mu_\alpha \to \mathcal{X}^\mu_\alpha,\quad U(\psi,u_0)=(u,u(\mu))
\end{equation*}
where $u$ is the weak solution to \eqref{eq1}  for the function $\psi$ and initial condition $u_0$. Then, under the assumptions of Lemma \ref{l2b}, it follows that $U$ is a continuous mapping in $\mathcal{X}^\mu_\alpha$.
\end{theorem}

\begin{proof}
By Lemma \ref{l2b} and Lemma \ref{l5b} this mapping is well--defined. Let $(\psi^i,u_0^i)\in \mathcal{X}^\mu_\alpha$ and let $u^i$ for $i=1,\,2$ be the associated weak solutions.
Then, the distributional derivative exists
\begin{align*}
  \frac{d (u^1(t)-u^2(t))}{dt} & =A(u^1(t)-u^2(t))+B(\psi^1(t)-\psi^2(t),u^1(t))+B(\psi^2(t),u^1(t)-u^2(t))
\end{align*}
Hence, for the distributional derivative of $\|u^1(t)-u^2(t)\|_\alpha^2$ we have
\begin{align}\label{eq6}
\begin{split}
  \frac{d\|u^1(t)-u^2(t)\|_\alpha^2}{dt} &+2\nu\| u^1(t)-u^2(t)\|_{1+\alpha}^2\le 2c\|\psi^1(t)-\psi^2(t)\|_{1+\alpha}\|u^1(t)\|_{1+\alpha}\|u^1(t)-u^2(t)\|_{\alpha} \\
  & +2c\|\psi^2(t)\|_{1+\alpha}\|u^1(t)-u^2(t)\|_{1+\alpha}\|u^1(t)-u^2(t)\|_{\alpha},
\end{split}
\end{align}
where we have applied Lemma \ref{l2} with $s_1=1+\alpha,\,s_2=\alpha,\,s_3=-\alpha$ in order to estimate the terms related to the nonlinearity. From this inequality we derive
\begin{align*}
  \frac{d\|u^1(t)-u^2(t)\|_\alpha^2}{dt} & \le c^2 \|\psi^1(t)-\psi^2(t)\|_{1+\alpha}^2+\|u^1(t)\|_{1+\alpha}^2\|u^1(t)-u^2(t)\|_{\alpha}^2+\frac{c^2}{\nu}
  \|\psi^2(t)\|_{1+\alpha}^2\|u^1(t)-u^2(t)\|_{\alpha}^2.\\
\end{align*}
Applying Gronwall's lemma we obtain
\begin{align*}
  \|u^1(t)-u^2(t)\|_\alpha^2\le  & \|u^1(0)-u^2(0)\|_\alpha^2\exp\bigg(\int_0^t (\|u^1(r)\|_{1+\alpha}^2+\frac{c^2}{\nu}\|\psi^2(r)\|_{1+\alpha}^2) dr\bigg)\\
  &+
  c^2\int_0^t\exp\bigg(\int_s^t (\|u^1(r)\|_{1+\alpha}^2+\frac{c^2}{\nu}\|\psi^2(r)\|_{1+\alpha}^2)dr\bigg) \|\psi^1(s)-\psi^2(s)\|_{1+\alpha}^2 ds.
\end{align*}
Now integrating the inequality \eqref{eq6} from $0$ to $\mu$ we obtain
\begin{align*}
  \nu\|u^1-u^2\|_{L_2(0,\mu,V^{1+\alpha})}^2\le& \|u^1(0)-u^2(0)\|_\alpha^2 + c^2\| \psi^1-\psi^2\|_{L_2(0,\mu,V^{1+\alpha})}^2\\+ &\|u^1\|_{L_2(0,\mu,V^{1+\alpha})}^2\sup_{t\in [0,\mu]}\|u^1(t)-u^2(t)\|_\alpha^2
  +\frac{c^2}{\nu}\|\psi^2\|_{L_2(0,\mu,V^{1+\alpha})}^2\sup_{t\in [0,\mu]}\|u^1(t)-u^2(t)\|_\alpha^2.
\end{align*}
Now it suffices to consider a sequence $(\psi^n,u_0^n)_{n\in\NN} \subset \mathcal{X}^\mu_\alpha$ converging to $(\psi,u_0) \in \mathcal{X}^\mu_\alpha$, and let $(u^n)_{n\in\NN}$ be the associated sequence of weak solutions. Since $\|\psi^n\|_{L_2(0,\mu,V^{1+\alpha})}^2$ is bounded, by the last two inequalities we derive that $(u^n,u^n(\mu))_{n\in\NN}$ converges to $(u,u(\mu))$ in $\mathcal{X}^\mu_\alpha$, where $u$ is the weak solution related to the data $(\psi,u_0)\in \mathcal{X}^\mu_\alpha$.

\end{proof}

\section{The 3D Navier-Stokes Equations with Delay}\label{s4}
Let $\mu$ be a positive number. We consider the following version of the 3D Navier-Stokes equations with constant delay $\mu$:
\begin{align}\label{delay}
\begin{split}
  u^\prime(t,x) & +(u(t-\mu),\nabla)u(t)-\nu\Delta u(t)+\nabla p(t,x)=f(x), \\
   & {\rm div}\, u(t,x)=0,\quad u(0)=u_0(x),\quad u(\tau)=\phi(\tau),\quad \tau\in [-\mu,0).
\end{split}
\end{align}
Denote the solution of this equation depending on the time shift by $u^\mu$. Taking again the Helmholtz-projection into account we can formulate the equation as
 \begin{equation}\label{D2}
  \left\{
  \begin{aligned}
    du^\mu(t)&+(Au^\mu(t)+B(u^\mu(t-\mu), u^\mu(t)))dt=fdt ,\qquad&t \geq 0,&\\
    u^\mu(t)&=\phi (t),\qquad&t\in[-\mu,0),&\\
    u^\mu(0)&=u_{0}.
  \end{aligned}
  \right.
\end{equation}
\begin{definition}\label{def-deb}
Let $\mu>0$ and $\alpha>1/2$.
We are given $u_0\in {V^\alpha}, \phi \in L^2(-\mu, 0, V^{1+\alpha})$ and $f\in V^{{\alpha-1}}$.
We say that $u^\mu$ is a weak solution to system \eqref{D2}
on the time interval $[-\mu,T]$ if
\begin{align*}
u^\mu\in L_2(-\mu,T, V^{1+\alpha}),
\end{align*}
with $u^\mu(0)=u_0$, $u^\mu(t)=\phi (t)$ for $t\in[-\mu,0)$, and, given any $v\in V^{\alpha+1}$ and any test function $\varphi \in  C_0^\infty([0,T])$,
\begin{align}\label{var-v}
  -\int_0^T \langle u^\mu(r),v \rangle \varphi^\prime(r)dr  &+\nu\int_{0}^{T}\langle A^{1/2}u^\mu(r), A^{1/2} v\rangle \varphi(r) dr +\int_{0}^{T}\langle B(u^\mu(r-\mu),u^\mu(r)),v \rangle \varphi(r) dr
\nonumber\\
&= \int_{0}^{T}\langle f,v \rangle \varphi(r)dr.
\end{align}
\end{definition}

\begin{theorem}\label{t1}
Suppose that $\phi\in L_2(-\mu,0,V^{1+\alpha})$, $u_0\in V^\alpha$ and $f\in V^{\alpha-1}$. Then there exists a unique weak solution to \eqref{var-v} in the sense of Definition \ref{def-deb}. Furthermore, $u^\mu|_{[0,T]}\in C([0,T], V^\alpha)$ and, if $0\leq \gamma \leq 1/2$ and $s\geq 1$, we also obtain $u^\mu|_{[0,T]}\in L_\infty(0,T,V^\alpha)\cap C^\gamma([0,T],V^{-s})$, and $\frac{du^\mu}{dt}\in L_2(0,T,V^{\alpha-1})$.% and it is weakly continuous with values in $V^0$.
\end{theorem}

\begin{proof} Without loss of generality, assume that $T=k\mu$ for some $k\in \NN$.

In what follows, by induction we are going to construct a sequence $(u^\mu_n)_{n\in \NN}$ such that any $u_n^\mu$ is a weak solution to the delay problem (\ref{delay}) on $[-\mu,\mu]$, with initial condition $u_n^\mu(0)= u^\mu_{n-1}(\mu)$ and initial delay function $u_n^\mu(t-\mu)=u_{n-1}^\mu(t)$, for $t\in [0,\mu)$ (hence $t-\mu\in [-\mu,0)$).

{\bf Step 1:} Existence and uniqueness of the first element of the sequence, namely $u^\mu_1$, corresponding to $\phi\in L_2(-\mu,0,V^{1+\alpha})$ and $u_0\in V^\alpha$. For this first step we take
$$\psi(t):=u^\mu_1(t-\mu)=\phi(t-\mu) \quad \mbox{ for almost  all } t\in [0,\mu).$$
Due to the regularity of $\phi$, we have $\psi\in L_2(0,\mu,V^{1+\alpha})$ and therefore, as a consequence of Lemma \ref{l3}, Lemma \ref{l4} and Lemma \ref{l2b}, there exists a unique weak solution $u^\mu_1 \in L_2(-\mu,\mu,V^{1+\alpha}) \cap L_\infty(0,\mu,V^\alpha)\cap C^\gamma ([0,\mu], V^{-s})$ to (\ref{delay}) corresponding to $u_0$ and $\phi$.

Moreover, in virtue of Lemma \ref{l5b}, such solution fulfills $u^\mu_1 \in C([0,\mu],V^\alpha)$ and thus $u^\mu_1(\mu) \in V^\alpha$.\\

{\bf Step $i$:} Let us assume that we have already obtained the elements $u_j^\mu$ for $j=1,\cdots, i-1$. In particular, $u^\mu_{i-1} \in L_2(-\mu,\mu,V^{1+\alpha}) \cap L_\infty(0,\mu,V^\alpha) \cap C^\gamma ([0,\mu], V^{-s}) \cap C([0,\mu], V^\alpha)$, such that in particular $u^\mu_{i-1}(\mu) \in V^\alpha$.

We would like to obtain the new element of the sequence $u_i^\mu$, defined on $[-\mu,\mu]$ and with initial condition $u_i^\mu(0)=u^\mu_{i-1}(\mu)$ and delay function $u^\mu_i(t-\mu)=u^\mu_{i-1}(t)$ for almost all $t\in [0,\mu)$. Then, for almost all $t\in [0,\mu)$ choosing now
$$\psi(t):=u^\mu_i(t-\mu)=u^\mu_{i-1}(t) \in L_2(0,\mu,V^{1+\alpha}),$$
we obtain $u^\mu_{i} \in L_2(-\mu,\mu,V^{1+\alpha}) \cap L_\infty(0,\mu,V^\alpha)\cap C^\gamma ([0,\mu],V^{-s})$, as a direct result of Lemma \ref{l3}, Lemma \ref{l4} and Lemma \ref{l2b} (note that $u_i^\mu(0)\in V^\alpha$). Lemma \ref{l5b} implies $u^\mu_i \in C([0,\mu], V^\alpha)$.

Now it suffices to define the function $u$ given by
\begin{align}
\label{sol}
u^\mu(t)=\left\{\begin{array}{lll}
    \phi(t)&\mbox{ if}& t\in[-\mu,0),\\
           u_0&\mbox{ if}& t=0,\\
        u^\mu_1(t)&\mbox{ if}& t\in [0,\mu],\\
     u^\mu_2(t-\mu)&\mbox{ if}& t\in [\mu,2\mu],\\
     \vdots&&\vdots \\
      u^\mu_k(t-(k-1)\mu)&\mbox{ if}& t\in [(k-1)\mu,T].
    \end{array}
    \right.
\end{align}
Then $u^\mu\in L_2(-\mu,T,V^{1+\alpha})$, $u^\mu|_{[0,T]}\in L_\infty(0,T,V^\alpha)\cap C^\gamma([0,T],V^{-s})$.

It remains to prove that $u^\mu$ given by (\ref{sol}) satisfies (\ref{var-v}). To simplify the presentation, let us assume that only the two first pieces $u^\mu_1$ and $u^\mu_2$ come into play.

Take a test function $\varphi \in C_0^\infty([0,2\mu])$. Since $u^\mu_1$ is a solution on $[-\mu,\mu]$ corresponding to the initial delay function $\phi$, thanks to (\ref{gal1}) and taking into account that $\varphi(0)=0$, we obtain
\begin{align}\label{5}
\begin{split}
-\int_0^\mu \langle u_1^\mu(r),v \rangle \varphi^\prime(r)dr&+\nu\int_{0}^{\mu}\langle A^{1/2}u_1^\mu(r), A^{1/2} v\rangle \varphi(r) dr\\&+\int_{0}^{\mu}\langle B(\phi(r-\mu),u_1^\mu(r)),v \rangle \varphi(r) dr-\int_{0}^{\mu}\langle f,v \rangle \varphi(r)dr\\
&=-(u^\mu_1(\mu),v)\varphi(\mu).
\end{split}
\end{align}
On the other hand, since $u_2^\mu$ is a solution on $[-\mu, \mu]$, such that $u_2^\mu(t-\mu)=u_{1}^\mu(t)$, for almost all $t\in [0,\mu)$, for a test function $\hat \varphi\in C^\infty([0,\mu])$ to be determined later, thanks to (\ref{gal1}) we have
\begin{align*}
\begin{split}
-\int_0^\mu \langle u_2^\mu(r),v \rangle \hat \varphi^\prime(r)dr&+\nu\int_{0}^{\mu}\langle A^{1/2}u_2^\mu(r), A^{1/2} v\rangle \hat \varphi(r) dr\\&+\int_{0}^{\mu}\langle B(u_1^\mu(r),u_2^\mu(r)),v \rangle \hat \varphi(r) dr-\int_{0}^{\mu}\langle f,v \rangle \hat \varphi(r)dr\\
&=(u^\mu_2(0),v)\hat \varphi(0)-(u^\mu_2(\mu),v)\hat \varphi(\mu),
\end{split}
\end{align*}
or equivalently
\begin{align*}
-\int_\mu^{2\mu} & \langle u^\mu_2(r-\mu),v \rangle \hat \varphi^\prime(r-\mu)dr+\nu\int_{\mu}^{2\mu}\langle A^{1/2}u_2^\mu(r-\mu), A^{1/2} v\rangle \hat \varphi(r-\mu) dr\\
&+\int_{\mu}^{2\mu}\langle B(u_1^\mu(r-\mu),u_2^\mu(r-\mu)),v \rangle \hat \varphi(r-\mu) dr-\int_{\mu}^{2\mu}\langle f,v \rangle \hat \varphi(r-\mu)dr\\
&=(u^\mu_2(0),v)\hat \varphi(0)-(u^\mu_2(\mu),v)\hat \varphi(\mu),
\end{align*}
hence, taking as test function $\hat \varphi(r-\mu):=\varphi(r), \, r\in [\mu, 2\mu]$, since  then $\hat \varphi(\mu)=\varphi(2\mu)=0$ we have
\begin{align}\label{6}
\begin{split}
-\int_\mu^{2\mu} \langle u^\mu_2(r-\mu),v\rangle \varphi^\prime(r)dr&+\nu\int_{\mu}^{2\mu}\langle A^{1/2}u_2^\mu(r-\mu), A^{1/2} v\rangle \varphi(r) dr\\
&+\int_{\mu}^{2\mu}\langle B(u_1^\mu(r-\mu),u_2^\mu(r-\mu)),v \rangle \varphi(r) dr-\int_{\mu}^{2\mu}\langle f,v \rangle \varphi(r)dr\\
&=(u^\mu_2(0),v) \varphi(\mu),
\end{split}
\end{align}
therefore, summing (\ref{5}) and (\ref{6}), taking into account the definition of $u^\mu$, since by construction $u^\mu_2(0)=u^\mu_1(\mu)$ we obtain
\begin{align*}
  -\int_0^{2\mu} \langle u^\mu(r),v \rangle \varphi^\prime(r)dr  &+\nu\int_{0}^{2\mu}\langle A^{1/2}u^\mu(r), A^{1/2} v\rangle \varphi(r) dr +\int_{0}^{2\mu}\langle B(u^\mu(r-\mu),u^\mu(r)),v \rangle \varphi(r) dr
\nonumber\\
&= \int_{0}^{2\mu}\langle f,v \rangle \varphi(r)dr,
\end{align*}
and this completes the proof.

\end{proof}

\begin{remark}
As we have shown, under the conditions of Theorem \ref{t1}, the weak solution $u^\mu \in L_2(-\mu,T,V^{1+\alpha})\cap L_\infty(0,T,V^\alpha)\cap C^\gamma([0,T],V^{-s})\cap C([0,T],V^{\alpha})$, with $0\leq \gamma \leq 1/2$, $s\geq 1$ and $\alpha >1/2$. However, the estimates in the previous spaces are not uniform in $\mu$. Nevertheless, the estimates in $L_2(-\mu,T,V^{1})$ and $L_\infty(0,T,V^0)$ are uniform in $\mu$, which follows by the estimates obtained in Lemma \ref{l3}. For a uniform estimate of the solution in the space of H\"older continuous functions see the Lemma \ref{l6}. Let us emphasize that, in order to obtain uniform estimates in $C^\gamma([0,T],V^{-s})$, we have to pay the price of replacing the space $V^{-s}$, $s\geq 1$, with $V^{-s}$ being $s> 3/2$. This uniform estimate will be used in Section \ref{s6} to pass to the limit when $\mu\to 0$.
\end{remark}

\begin{lemma}\label{l6}
Under the conditions of Theorem \ref{t1}, the weak solution $u$ is uniformly bounded in $\mu$ in the space $C^\gamma([0,T],V^{-s})$, for $0\leq \gamma \leq 1/2$ and $s> 3/2$.
\end{lemma}
\begin{proof}
Let us assume that $s> \frac{3}{2} $ and $0\leq \tau<t\leq T$, then
\begin{align*}
  \|u(t)-u(\tau)\|_{-s}\le&\nu (t-\tau)^\frac12\bigg(\int_0^T\|Au(r)\|_{-s}^2dr\bigg)^\frac12\\
  &+(t-\tau)^\frac12\bigg(\int_0^T\| B(u((r-\mu),u(r))\|_{-s}^2dr\bigg)^\frac12.
\end{align*}
On the one hand, similar to the proof in Lemma \ref{l3}, the estimate of $ \int_0^T\|Au(r)\|_{-s}^2dr $ follows. Furthermore, by Lemma \ref{l1}, with the choice $s_1=s_2=0,\,s_3=s>3/2$, we obtain
\begin{align*}
  \|B(u(r-\mu),u(r))\|_{-s}\le c\|u(r-\mu)\|_{1}\|u(r)\|_{0}.
\end{align*}
Putting all the estimates together completes the proof.
\end{proof}

\section{ The semigroup generated by the solution of (\ref{D2})}\label{s5}

From now on, denote by $\mathcal Y^\mu_\alpha$ the space $L_2(-\mu,0,V^{1+\alpha})\times V^\alpha$.
\begin{lemma}
The weak solution $u$ to (\ref{D2}) corresponding to the initial data $(\phi,u_0)\in \mathcal Y^\mu_\alpha$ defines a  semigroup $S(t):\mathcal Y^\mu_\alpha \rightarrow \mathcal Y^\mu_\alpha$, for $t\geq 0$,  given by
\begin{equation}\label{sem}
S(t)(\phi,u_0)=(u_t^\mu, u^\mu(t)),
\end{equation}
where $u_t^\mu$ is the segment function defined by $u_t^\mu(s)=u^\mu(t+s)$, $s\in (-\mu,0)$.
\end{lemma}
\begin{proof} To simplify the proof, we drop off the superscript $\mu$.

First of all, it is straightforward to check that $S(0)={\rm Id}_{\mathcal Y^\mu_\alpha}$.\\
In order to prove the semigroup property, namely, $S(t+\tau)(\phi,u_0)=S(t) (S(\tau)(\phi,u_0))$, for $t,\, \tau \geq 0$, it is sufficient to consider $u^1$ to be the solution to (\ref{D2}) on $[-\mu,\tau]$, with initial data $(\phi,u_0)\in \mathcal Y^\mu_\alpha$, and $u^2$ to be the solution to (\ref{D2}) on $[-\mu,t]$, with initial data $(u^1_\tau,u^1(\tau))\in \mathcal Y^\mu_\alpha$. Now, for almost all $\theta\in [-\mu,\tau+t]$ we define the function $u$ as follows
\begin{align}  \label{u}
u(\theta)=\left\{\begin{array}{lll}
    u^1(\theta)&\mbox{ if}& \theta\in[-\mu,\tau],\\
        u^2(\theta-\tau)&\mbox{ if}& \theta\in [\tau,\tau+t].\\
    \end{array}
    \right.
\end{align}
Let us observe that if $\theta\in [\tau-\mu,\tau]$ then $u^1(\theta)=u^2(\theta-\tau)$, or equivalently, $u^1(\tau+s)=u^1_\tau(s)=u^2(s)$, for almost all $s\in [-\mu,0]$, which follows from the fact that the initial delay function for  $u^2$ is given by $u^1_\tau$.

In order to prove the semigroup property we simply need to show that $u$ is a solution on $[-\mu, t+\tau]$, which is similar to the concatenation property proved in Theorem \ref{t1}, hence we omit the proof.

\end{proof}

\begin{remark}\label{r1}
There exists a clear relation between the mapping $U$ given in Theorem \ref{t3} and the semigroup $S(\cdot)$. Let us recall that $\mathcal{X}^\mu_\alpha=L_2(0,\mu,V^{1+\alpha})\times V^\alpha$ and $U(\psi,u_0)=(u,u(\mu))$ for $(\psi,u_0)\in \mathcal{X}^\mu_\alpha$.

If now $(\phi,u_0)\in \mathcal Y^\mu_\alpha$, then
\begin{align*}
S(\mu)(\phi(\cdot),u_0)=(u_\mu(\cdot),u(\mu))=(u(\mu+\cdot),u(\mu))=U(\psi(\cdot),u_0),
\end{align*}
which relies on the fact that if $\phi\in L_2(-\mu,0,V^{1+\alpha})$, then $\psi(\cdot):=\phi(\cdot-\mu)$ is such that $\psi \in L_2(0,\mu,V^{1+\alpha})$.

\end{remark}

Note that since confusion is not possible, we have dropped off the superscript $\mu$ from  the notation of the solution $u^\mu$.
\begin{lemma} \label{l7}
For $t\geq 0$, the semigroup $S(t)$ given by (\ref{sem}) is continuous on $\mathcal Y^\mu_\alpha$.
\end{lemma}

\begin{proof}
Given $t\geq 0$, there exists $k\in \NN$ such that $t\in [k\mu, (k+1)\mu]$. Due to the relationship between the mapping $U$ and the semigroup $S$, see Remark \ref{r1}, we have that
\begin{align*}
S(t)(\phi(\cdot),u_0) &=S(t-k\mu) S(k\mu)(\phi(\cdot),u_0)=S(t-k\mu) (S(\mu)\circ \cdots \circ S(\mu))(\phi(\cdot),u_0)\\
&= S(t-k\mu)(U\circ \cdots \circ U) (\psi(\cdot),u_0).
\end{align*}
We proved previously in Theorem \ref{t3} that $U$ is a continuous map in $\mathcal{X}^\mu_\alpha$. Since $0\leq t-k\mu\leq \mu$, it is sufficient to prove the continuity of $S$ on the time interval $[0,\mu]$.

Let us consider $S(t)$ for $t\in [0,\mu]$, let $(\phi^i,u_0^i)\in \mathcal{Y}^\mu_\alpha$, and let $u^i$ for $i=1,\,2$ be the associated weak solutions.
Note that because $t\in[0,\mu]$ the nonlinear term can be handle as follows:
\begin{align*}
B(u^1(t-\mu),u^1(t))-B(u^2(t-\mu),u^2(t))&=B(u^2(t-\mu),u^1(t)-u^2(t))+B(u^1(t-\mu)-u^2(t-\mu),u^1(t))\\
&=B(\phi^2(t-\mu),u^1(t)-u^2(t))+B(\phi^1(t-\mu)-\phi^2(t-\mu),u^1(t)).
\end{align*}

Then, in a similar way to the proof of Theorem \ref{t3}, we arrive at
\begin{align*}
  \|u^1(t)-u^2(t)\|_\alpha^2\le  & \|u^1(0)-u^2(0)\|_\alpha^2\exp\bigg(\int_0^t(\|u^1(r)\|_{1+\alpha}^2+\frac{c^2}{\nu}{\|\phi^2(r-\mu)\|_{1+\alpha}^2})dr\bigg)\\
  +&
  c^2\int_0^t\exp\bigg(\int_s^t(\|u^1(r)\|_{1+\alpha}^2+\frac{c^2}{\nu}\|\phi^2(r-\mu)\|_{1+\alpha}^2)dr\bigg) {\|\phi^1(s-\mu)-\phi^2(s-\mu)\|_{1+\alpha}^2} ds.
\end{align*}
and
\begin{align*}
  \nu\|u^1-u^2\|_{L_2(0,\mu,V^{1+\alpha})}^2& \leq \| u^1(0)-u^2(0)\|_\alpha^2 + c^2\|\phi^1-\phi^2\|_{L_2(-\mu,0,V^{1+\alpha})}^2\\
  &+ \|u^1\|_{L_2(0,\mu,V^{1+\alpha})}^2\sup_{t\in [0,\mu]}\|u^1(t)-u^2(t)\|_\alpha^2+ \frac{c^2}{\nu}\|\phi^2\|_{L_2(-\mu,0,V^{1+\alpha})}^2\sup_{t\in [0,\mu]}\|u^1(t)-u^2(t)\|_\alpha^2.
\end{align*}
From this last inequality, we obtain a suitable inequality for $ \|u^1_t-u^2_t\|^2_{L_2(-\mu,0,V^{1+\alpha})}$, since
\begin{align*}
 \|u^1_t-u_t^2\|^2_{L_2(-\mu,0,V^{1+\alpha})}&= \int_{-\mu}^0  \|u^1(t+r)-u^2(t+r)\|^2_{1+\alpha} dr \\
  &= \int_{t-\mu}^0  \|u^1(r)-u^2(r)\|^2_{1+\alpha} dr+\int_0^t  \|u^1(r)-u^2(r)\|^2_{1+\alpha} dr\\
  &\leq  \|\phi^1-\phi^2 \|^2_{L_2(-\mu,0,V^{1+\alpha})}+ \| u^1-u^2\|^2_{L_2(0,\mu,V^{1+\alpha})}.
\end{align*}
Now it suffices to consider a sequence $(\phi^n,u_0^n)_{n\in\NN} \subset \mathcal Y^\mu_\alpha$ converging to $(\phi,u_0) \in \mathcal Y^\mu_\alpha$. Let $(u^n)_{n\in\NN}$ be the associated sequence of weak solutions corresponding to $(\phi^n,u_0^n)_{n\in\NN}$, while that $u$ denotes the solution of (\ref{var-v}) corresponding to $(\phi,u_0)$. Since $\|\phi^n\|_{L_2(-\mu,0,V^{1+\alpha})}^2$ is bounded, by the last inequalities we derive that $(u^n_t,u^n(\mu))_{n\in\NN}$ converges to $(u_t,u(\mu))$ in $\mathcal Y^\mu_\alpha$.

\end{proof}

\section{The convergence $\mu\rightarrow 0$}\label{s6}
In this section we would like to study the behavior of the solution $\um$ when the delay $\mu$ tends to zero. To be a bit more precise, we will prove that one can extract a subsequence, still denoted $\um$, that converges to $u$ in some appropriate sense and that $u$ is a weak solution to the 3D Navier-Stokes equations associated to the initial condition $u_{0}$.

\begin{theorem}\label{mu-conv}
Given $T>0$, let $\{\mu_n\}_{n\in \NN}$ a sequence such that $\lim_{n\to \infty} \mu_n=0$. Assume that $\phi^{\mu_n} \in L_2(-{\mu_n},0,V^{1+\alpha})$, such that $\sup_{n\in \NN}\|\phi^{\mu_n}\|_{L_2(-{\mu_n},0,V^{1+\alpha})}<\infty$. Let $\{u^{\mu_n}_0\}_{n\in \NN} \subset V^\alpha$ , such that w--$\lim_{n\to \infty} u_0^{\mu_n}=u_0$ in $V^0$, and $f\in V^{\alpha-1}$. If $u^{\mu_n}$ denotes the corresponding weak solution of \eqref{D2} on $[-\mu_n,T]$ with $u^{\mu_n}(0)=u^{\mu_n}_0$ and $u^{\mu_n}(t)=\phi^{\mu_n}(t)$ for almost all $t\in [-\mu_n,0)$, then we can extract a subsequence that converges in $L_2(0,T,V^0)\cap C([0,T],V^{-s})$, $s>3/2$, to a limit $u$ when $n$ goes to infinity. Moreover, the limit
$u$ is a weak solution of the 3D Navier-Stokes equations, namely, given any $v\in V^1$ and any test function $\varphi \in  C_0^\infty([0,T])$ then
\begin{align}\label{weak-nse}
  -\int_0^T \langle u(r),v \rangle \varphi^\prime(r)dr  &+\nu\int_{0}^{T}\langle A^{1/2}u(r), A^{1/2}v\rangle  \varphi(r)dr +\int_{0}^{T}\langle B(u(r),u(r)),v \rangle \varphi(r) dr
\nonumber\\
&= \int_{0}^{T}\langle f,v \rangle \varphi(r)dr
\end{align}
with $u(0)=u_0\in V^0$.%, the weak limit of $u_0^\mu$  in $V^0$.{\color{red} to be continued}
\end{theorem}

\begin{proof}
Using the results of the previous sections,  we have that the sequence $u^{\mu_n}$ is uniformly bounded in $L_\infty(0,T,V^0)\cap L_2(0,T,V^1)\cap C^\gamma([0,T],V^{-s})$, for $0\leq \gamma\leq1/2$ and $s> 3/2$, which is compactly embedded in
$L_2(0,T,V^0)\cap C([0,T],V^{-s})$. Hence, we deduce that there exists a subsequence still denoted $u^{\mu_n}$ such that
\begin{align*}
  &\lim_{n\to \infty}\|u^{\mu_n}-u\|_{L_2(0,T,V^0)}=0,\\
  w-&\lim_{n\to \infty}u^{\mu_n}=u\quad \text{in }L_2(0,T,V^1),\\
  &\lim_{n\to \infty}\|u^{\mu_n}(t)-u(t)\|_{-s}=0,\quad \forall t\in [0,T],\\
  &\lim_{n\to \infty} \langle u^{\mu_n}(t)-u(t), v \rangle =0,\quad \forall t\in [0,T], \, v \in V^s, \, s>3/2.
\end{align*}
In particular,
\begin{align*}
0=\lim_{n\to \infty}\|u^{\mu_n}(0)-u(0)\|_{-s}=\lim_{n\to \infty}\|u^{\mu_n}_{0}-u(0)\|_{-s}.
\end{align*}
Since by assumption  $\lim_{n\to \infty}(u^{\mu_n}_0,v)=(u_0,v)$, $ \forall v\in V^0$,  we deduce that
$u(0)=u_0\in V^0$.

From the previous convergences, it is straightforward to obtain the following convergences when $n$ tends to infinity:

$$ -\int_0^T \langle u^{\mu_n}(r),v \rangle \varphi^\prime(r)dr  \rightarrow  -\int_0^T \langle u(r),v \rangle \varphi^\prime(r)dr  $$
and
$$ \int_{0}^{T}\langle A^{1/2}u^{\mu_n}(r), A^{1/2}v \rangle \varphi(r) dr \rightarrow
\int_{0}^{T}\langle A^{1/2}u(s), A^{1/2}v \rangle \varphi(r) dr.$$

For the nonlinear term, we can consider the following splitting:
\begin{align*}
&\int_{0}^{T}\bigg(\langle B(u^{\mu_n}(r-\mu_n),u^{\mu_n}(r)), v\rangle-\langle B(u(r),u(r)), v\rangle \bigg)\varphi(r)dr\\
=&\int_{0}^{T}\langle B(u^{\mu_n}(r-\mu_n), u^{\mu_n}(r)-u(r)), v\rangle \varphi(r)dr\\
&+\int_{\mu_n}^{T}\langle B(u^{\mu_n}(r-\mu_n)-u(r-\mu_n), u(r)), v\rangle \varphi(r)dr\\
&+\int_{\mu_n}^{T}\langle B(u(r-\mu_n)-u(r),u(r)),v\rangle \varphi(r)dr\\
&+\int_{0}^{\mu_n}\langle B(u^{\mu_n}(r-\mu_n)-u(r),u(r)), v\rangle \varphi(r)dr=:I_1+I_2+I_3+I_4.
\end{align*}

Now it suffices to apply in suitable ways Lemma \ref{l2}. Hence, choosing $s>3/2$,
\begin{align*}
|I_1|&\leq \int_{0}^{T}\|u^{\mu_n}(r)-u(r)\|_{0}\|u^{\mu_n}(r-\mu_n)\|_{1}\|v\|_{s} \varphi(r)dr\\
&\leq \|v\|_{s} \|u^{\mu_n}\|_{L_2(-\mu_n,T,V^1)} \bigg(\int_{0}^{T}\|u^{\mu_n}(r)-u(r)\|^2_{0} dr \bigg)^{1/2}\left(\sup_{r\in[0,T]}|\varphi(r)| \right),
\end{align*}
that tends to zero when $n \to \infty$. Also, for $s>3/2$,
\begin{align*}
|I_2|&\leq \int_{\mu_n}^{T}\|u^{\mu_n}(r-\mu)-u(r-\mu_n)\|_{0}\|u(r)\|_{1}\|v\|_{s} \varphi(r)dr\\
&\leq \|v\|_{s} \|u\|_{L_2(0,T,V^1)} \bigg(\int_{0}^{T-\mu_n}\|u^{\mu_n}(r)-u(r)\|^2_{0} dr\bigg)^{1/2} \left(\sup_{r\in[0,T]}|\varphi(r)| \right)\\
&\leq \|v\|_{s} \|u\|_{L_2(0,T,V^1)} \bigg(\int_{0}^{T}\|u^{\mu_n}(r)-u(r)\|_{0}^2 dr\bigg)^{1/2} \left(\sup_{r\in[0,T]}|\varphi(r)| \right),
\end{align*}
which converges to zero when $n \to \infty$.

Now, in order to prove the convergence of $I_3$ we use the weakly continuity of $u$ in $V^0$, see Lemma 1.4, Chapter 3, Page 263 in \cite{Temam}. Indeed, we can express $I_3$ equivalently as
\begin{align*}
I_3=\int_{0}^{T} 1_{[\mu_n,
T]}(r) b(u(r-\mu_n)-u(r),v,u(r)) \varphi(r)dr.
\end{align*}
The integrand of that expression can be estimated by $2\|u\|_{L_\infty(0,T,V^0)} \|u(r)\|_1\|v\|_s |\varphi(r)|$, for $s>3/2$, which is integrable. 
\begin{align*}
\lim_{n \to \infty} |I_3|&=\int_{0}^{T} \lim_{n \to \infty} 1_{[\mu_n,T]}(r) b(u(r-\mu_n)-u(r),v,u(r)) \varphi(r)dr\\
&\leq \int_{0}^{T} \lim_{n \to \infty} 1_{[\mu_n,T]}(r) \sum_{i,j=1}^3 \left( \int_{\mathbb T_L^3}  (u_j(r-\mu_n,x)-u_j(r,x))\frac{\partial v_i(x)}{\partial x_j}  u_i(r,x) dx \right) \varphi(r)dr\\
&=\int_{0}^{T} \lim_{n \to \infty} 1_{[\mu_n,T]}(r) ( u(r-\mu_n)-u(r),\pi\kappa(r)   ) dr
\end{align*}
and this last limit is zero due to the weak continuity of $u$ in $V^0$ by choosing $\kappa_j(r,x)=\sum_{i}\frac{\partial v_i(x)}{\partial x_j}  u_i(r,x) \varphi(r) \in L_2(\TT_L^3)$, for almost every $r\in[0,T]$ by Lemma \ref{l1}. $\pi$ denotes the Helmholtz projection. In particular, we have for almost all $r$
\begin{equation*}
  \lim_{n\to\infty}( u(r-\mu_n)-u(r),\pi\kappa(r)   )=\lim_{n\to\infty}( u(r-\mu_n)-u(r),\kappa(r)   )_{L_2(\TT_L^3)^3}=0.
\end{equation*}
To ensure that $\pi\kappa(r)\in V^0$  for almost all $r$, it is  enough that $v\in V^s$ and $u(r)\in V^1$ for almost all $r$. The dominated convergence theorem
gives the convergence of $I_3$ to zero.
Finally we have
\begin{align*}
|I_4|&\leq \int_{0}^{\mu_n}\|u^{\mu_n}(r-\mu_n)-u(r)\|_{1}\|u(r)\|_{0}\|v\|_{s} \varphi(r)dr\\
&\leq \int_{0}^{\mu_n}\|\phi^{\mu_n}(r-\mu_n)\|_{1}\|u(r)\|_{0}\|v\|_{s} \varphi(r)dr+\int_{0}^{\mu_n}\|u(r)\|_{1}\|u(r)\|_{0}\|v\|_{s} \varphi(r)dr\\
&\leq {\mu_n}^{1/2} \left(\sup_{r\in[0,T]}|\varphi(r)| \right) \|v\|_{s}  \|u\|_{L_\infty(0,T,V)}  \left( \|\phi^{\mu_n}\|_{L_2(-\mu_n,0,V^1)} +  \|u\|_{L_2(0,T,V^1)}\right),
\end{align*}
hence the proof is complete.

\end{proof}

{\bf Aknowledgement:} This paper was partially finished while the authors were visiting the Center for Mathematical Sciences in Wuhan (China). We would like to thank Prof. J. Duan and all the staff of the center for a very warm hospitality.

Hakima Bessaih was partially supported by NSF grant DMS-1418838.

%\bibliography{3dnse_bjoern}
%\bibliographystyle{plain}
\end{document}